\documentclass{amsart}

\usepackage{amsthm,amstext,amsfonts}
\usepackage{amsmath}
\usepackage{amssymb}
\usepackage[ansinew]{inputenc}
\usepackage{amscd}
\usepackage[ansinew]{inputenc}
\usepackage[mathscr]{euscript}
\usepackage{color}
\usepackage[all]{xy}

\newtheorem{thm}{Theorem}[section]
\newtheorem{cor}[thm]{Corollary}

\newtheorem{prop}[thm]{Proposition}

\newtheorem{clm}[thm]{Claim}

\theoremstyle{definition}
\newtheorem{ex}[thm]{Example}

\theoremstyle{definition}
\newtheorem{defn}[thm]{Definition}

\theoremstyle{definition}
\newtheorem{rem}[thm]{Remark}

\theoremstyle{definition}

\def\Q{\mathbb Q}
\def\C{\mathbb C}

\def\dim{\operatorname{dim}}

\def\O{\mathcal O}

\def\m{\mathbf m}

\def\({\left(}
\def\){\right)}

\def\geq{\geqslant}
\def\leq{\leqslant}

\newcommand{\smallfrac}[2]{{\textnormal{\small$\frac{#1}{#2}$}}}

\makeatletter
\newcommand{\tpitchfork}{%
  \vbox{
    \baselineskip\z@skip
    \lineskip-.52ex
    \lineskiplimit\maxdimen
    \m@th
    \ialign{##\crcr\hidewidth\smash{$-$}\hidewidth\crcr$\pitchfork$\crcr}
  }%
}

\begin{document}

%\title[short text for running head]{full title}
\title[Quasihomogeneous Functions on Analytic Varieties]{Bruce-Roberts Numbers and Quasihomogeneous Functions on Analytic Varieties}


%    author one information
\author{C. Bivi\`a-Ausina}
\address{
Institut Universitari de Matem\`atica Pura i Aplicada,
Universitat Polit\`ecnica de Val\`encia,
Cam\'i de Vera, s/n,
46022 Val\`encia, Spain}
\email{carbivia@mat.upv.es}
\thanks{The work of the first author was supported by Grant PID2021-124577NB-I00 funded by MCIN/AEI/10.13039/501100011033 and by ``ERDF A way of making Europe".}

%    author two information
\author{K. Kourliouros}
\address{Imperial College London, Department of Mathematics, 
180 Queen's Gate, South Kensington Campus, London SW7 2AZ, United Kingdom}
\email{k.kourliouros@gmail.com}
\thanks{The work of the second author was partially supported by the UK EPSRC [EP/W009455/1], and also by the Research Foundation of S\~ao Paulo (FAPESP) Grant No. 2017/2355-9.}

\author{M. A. S. Ruas}
\address{
Instituto de Ci\^encias Matem\'aticas e de Computa\c{c}\~ao,
Universidade de S\~ao Paulo,
Av. Trabalhador S\~ao-carlense, 400,
13566-590 S\~ao Carlos, SP, Brazil}
\email{maasruas@icmc.usp.br}
\thanks{The work of the third author was partially supported by FAPESP Proc. 2019/21181-0 and CNPq Proc. 305695/2019-3.}

%    \subjclass and \keywords are not used by JAG.

\date{}

%\dedicatory{}

    % "Communicated by" -- provide editor's name; required.
%\commby{}

\begin{abstract}
Given a germ of an analytic variety $X$ and a germ of a holomorphic function $f$ with a stratified isolated singularity with respect to the logarithmic stratification of $X$, we show that under certain  conditions on the singularity type of the pair $(f,X)$, the following relative analog of the well known K. Saito's theorem holds true: equality of the relative Milnor and Tjurina numbers of $f$ with respect to $X$ (also known as Bruce-Roberts numbers) is equivalent to the relative quasihomogeneity of the pair $(f,X)$, i.e. to the existence of a coordinate system such that both $f$ and $X$ are quasihomogeneous with respect to the same positive rational weights. 
\end{abstract}

\maketitle


\bibliographystyle{amsplain}

%%%%%%%%%%%%%%%%%%%%%%%%%%%%%%%%%%%%%%%%%%%%%%%%%%%%%%%%%%%%%%%%%%%%%%%%

%    Templates for common elements of a journal article; for additional
%    information, see the file Author_Handbook_Journals.pdf, included
%    in the JAG author package, and the amsthm user's guide, linded from
%    http://www.ams.org/tex/amslatex.html .

%    Section headings
\section{Introduction-Main Results}
\label{sec:1}

Given a holomorphic function germ $f:(\C^n,0)\rightarrow (\C,0)$ with an isolated singularity at the origin, its Milnor number is  classically defined as
\[\mu(f)=\dim_{\mathbb{C}}\frac{\O_n}{df(\Theta)}\]
where $\O_n$ is the ring of holomorphic function germs at the origin of $\C^n$, $\Theta \cong \O_n^n$ is the module of germs of vector fields (derivations),  and $df(\Theta)=J(f)$ is the ideal generated by the partial derivatives of $f$. By Milnor's theorem (c.f. \cite{Mil}, \cite{Pal}) this is exactly the rank of the middle homology group of the Milnor fiber of $f$, equal to the number of spheres in its bouquet decomposition.

Along with the Milnor number, one also defines the Tjurina number of $f$  
\[\tau(f)=\dim_{\mathbb{C}}\frac{\O_n}{df(\Theta)+\langle f\rangle}\]
 (where $\langle f\rangle\subset \O_n$ is the ideal generated by $f$), which is interpreted as the dimension of the base of a semi-universal deformation of the isolated hypersurface singularity  $Y=f^{-1}(0)$ defined by $f$.  

By definition $\mu(f)\geq \tau(f)$, and according to a well known theorem of K. Saito \cite{Sa1}, equality $\mu(f)=\tau(f)$ is equivalent to the quasihomogeneity of $f$, i.e. to the existence of a coordinate system $x=(x_1,\dots,x_n)$ and a vector of positive rational numbers $w=(w_1,\dots,w_n)\in \Q^n_+$ (the weights), such that $f$ can be written as
\[f(x)=\sum_{\langle w,m\rangle=1} a_mx^m\]
where $a_m\in \C$ and $x^m:=x_1^{m_1}\dots x_n^{m_n}$ are those monomials in the expansion of $f$ whose exponents $m=(m_1,\dots,m_n)\in \mathbb{N}^n$ belong to the affine hyperplane
\[\langle w,m\rangle=\sum_{i=1}^nw_im_i=1.\]
Notice that positivity of the weights $w_i$, $i=1,\dots,n$, forces $f$ to be polynomial in the above coordinates. 

After Saito's proof, there have been many other characterizations of quasihomogeneity of isolated hypersurface singularities, relating it with other invariants of the singularity ( c.f. \cite{Yau} and references therein). Moreover, Saito's result has been generalised for the case of isolated complete intersection singularities of positive dimension ({\textsc{icis}} for short) $X\subset (\C^n,0)$ by H. Vosegaard \cite{Vo}, a problem of substantial difficulty and of rather long history (c.f. \cite{Gr}, \cite{Gr0} for curves, \cite{Wahl} for surfaces and \cite{Vo1} for purely elliptic {\textsc{icis}} of dimension $\geq 2$): given an {\textsc{icis}} germ $X=h^{-1}(0)$, $h\in \O_n^m$, $n\geq m+1$, equality of its Milnor number $\mu(h)$ (i.e. the rank of the middle homology group of the Milnor fiber of $h$, c.f. \cite{Gr1}, \cite{Ha} and \cite{Looijenga}) and its Tjurina number $\tau(h)$ (i.e. the dimension of the base of a semi-universal deformation of $X$, c.f. \cite{Gr2} and \cite{Looijenga}), is equivalent to the quasihomogeneity of $X$, i.e. to the existence of a coordinate system $x=(x_1,\dots,x_n)$ and a vector of positive rational numbers $w=(w_1,\dots,w_n)\in \Q^n_+$,  such that the ideal of functions $I_X$ vanishing on $X$ admits a system of quasihomogeneous generators with respect to the weights $w$, that is
\[I_X=\langle h_1,\dots,h_m\rangle\]
\[h_i(x)=\sum_{\langle w,s\rangle=d_i} b_{s,i}x^s, \hspace{0.2cm} i=1,\dots,m\]
where $b_{s,i}\in \C$, and $d_i\in \Q_+$ are the quasihomogeneity degrees of the $h_i$'s. 

In contrast to Saito's proof, Vosegaard's proof is highly non-trivial, since the difference $\mu(h)-\tau(h)$ of the corresponding Milnor and Tjurina numbers of an {\textsc{icis}} $X=h^{-1}(0)$ admits no simple expression, as in the hypersurface case, but involves instead several invariants coming from the mixed Hodge structure of the link $X\setminus \{0\}$, and the resolution of singularities of $X$ (c.f. \cite{LS} and \cite{Vo}).

Away from the {\textsc{icis}} case, the invariant characterization of quasihomogeneity for general analytic varieties $X\subset (\C^n,0)$ is problematic, at least in terms of numerical invariants generalising the Milnor and Tjurina numbers, which typically cease to exist (i.e. they are not finite). In fact, the simplest  possible characterization of quasihomogeneity of analytic sets, involves the module of the so-called logarithmic vector fields (as defined by K. Saito \cite{Sa2})
\[\Theta_X:=\big \{\delta \in \Theta : \delta (I_X)\subseteq I_X\big \}\]
i.e. those vector fields which are tangent to the smooth part of $X$. In particular, one may easily show (c.f. \cite{DJZ} and also Theorem \ref{log-pd} in Section \ref{sec:3} of the present paper) that the variety $X$ is quasihomogeneous in an appropriate coordinate system, if and only if there exists a logarithmic vector field $\delta \in \Theta_X$ which vanishes at the origin, $\delta(0)=0$, and has positive rational eigenvalues
\[\text{sp}(\delta)=w=(w_1,\dots,w_n)\in \Q^n_+\]
where we denote by $\text{sp}(\delta)$ the spectrum of $\delta$, i.e. the set of eigenvalues of its linear part $j^1\delta$, viewed as a linear operator in $\C^n$. In fact, both Saito's and Vosegaard's proof, rely on the fact that in case where $X$ is an isolated hypersurface singularity, or an {\textsc{icis}} respectively, then equality of the corresponding Milnor and Tjurina numbers provides exactly such a logarithmic vector field with the required positivity property on its eigenvalues.

In the present paper we give another generalisation of Saito's theorem which interpolates between the above cases and is relevant in relative singularity theory. We consider pairs $(f,X)$ where $X\subset (\C^n,0)$ is an arbitrary analytic variety, and $f\in \O_n$ is a function germ which has a stratified isolated singularity at the origin (in the sense of L\^e \cite{Le2}) with respect to the logarithmic stratification of $X$, as defined by K. Saito (c.f. \cite{Sa2} for the case of hypersurfaces, and also \cite{BR} for the more general case of arbitrary analytic varieties $X$). 

The number that naturally generalises the Milnor number in this situation is the following relative Milnor number
\[\mu_X(f)=\dim_{\C}\frac{\O_n}{df(\Theta_X)}\]
where $df(\Theta_X)=J_X(f)$ is the ideal generated by the (Lie) derivatives of $f$ along logarithmic vector fields of $X$. This number has been consider by many authors, starting probably from the works of V. I. Arnol'd \cite{A1} in the case where $X$ is a smooth divisor, by O. V. Lyashko \cite{Ly} in the case where $X$ is an isolated hypersurface singularity, and later on by J. W. Bruce and R. M. Roberts \cite{BR} for the more general case of arbitrary analytic varieties $X$ (and many others which is impossible to cite). Recently (c.f. \cite{ART2012}, \cite{NOT2011}, \cite{BiviaRuas}, \cite{Gru}, \cite{Kour}, \cite{NOPT2020}-\cite{LNOTarxiv}, \cite{NT1}) it has been called the {\it Bruce-Roberts Milnor number} of $(f,X)$ (or of $f$ with respect to $X$). We adopt this terminology here as well.

We remark (c.f \cite[pp. 64]{BR}) that finiteness of the Bruce-Roberts Milnor number $\mu_X(f)<\infty$ is equivalent to the finite $\mathcal{R}_X$-determinacy of the function $f$ (i.e. finite determinacy under diffeomorphisms preserving the variety $X$), and to the existence of an $\mathcal{R}_X$-versal deformation for $f$ as well. As it was mentioned earlier, it is also equivalent to the function $f$ being a submersion on each logarithmic stratum of $X$, except possibly at the origin. 

In analogy with the Bruce-Roberts Milnor number, one may also define the {\it Bruce-Roberts Tjurina number} of the pair $(f,X)$ (c.f. \cite{Ah2} and \cite{BiviaRuas})
\[\tau_X(f)=\dim_{\C}\frac{\O_n}{df(\Theta_X)+\langle f\rangle}\]
which encodes the infinitesimal deformations of $f$ under $\mathcal{K}_X$-equivalence, i.e. under diffeomorphisms preserving $X$ and multiplication of $f$ by units in $\O_n$. In contrast to the Bruce-Roberts Milnor number, the properties of the Bruce-Roberts Tjurina number are much less studied in the literature, with an exception being the recent work  \cite{BiviaRuas} of the first and third authors, where they show, inspired by a result of Y. Liu \cite{Liu}, that the quotient $\mu_X(f)/\tau_X(f)$ is always bounded from above by the smallest integer $r$ such that $f^r\in df(\Theta_X)$. 

The present paper is motivated by the natural question as to characterise those pairs $(f,X)$ for which $r=1$, i.e. such that the equality $\mu_X(f)=\tau_X(f)$ holds. By definition, this is equivalent to $f\in df(\Theta_X)$, i.e. to the existence of a logarithmic vector field $\delta \in \Theta_X$ such that $\delta(f)=f$. Notice that by the obvious inclusion $\Theta_X\subset \Theta$, equality $\mu_X(f)=\tau_X(f)$ immediately implies $\mu(f)=\tau(f)$, and thus by Saito's theorem \cite{Sa1}, we know that there always exist a coordinate system such that $f$ can be reduced to a quasihomogeneous polynomial with positive rational weights. Despite this fact, we don't know if in these coordinates the variety $X$ will also be quasihomogeneous with respect to the same weights, or if it will be quasihomogeneous at all, and this is exactly the problem that we want to address here. The following definition will be useful throughout the paper.

\begin{defn}
\label{def-rel-quasi-1}
A pair $(f,X)$ in $(\C^n,0)$ will be called relatively quasihomogeneous if there exists a vector of positive rational numbers $w=(w_1,\dots,w_n)\in \Q_+^n$, a system of coordinates $x=(x_1,\dots,x_n)$ and a system of generators $\langle h_1,\dots,h_m\rangle=I_X$ of the ideal of functions vanishing on $X$, such that
\begin{align*}
f(x) &=\sum_{\langle w,m\rangle=1}a_mx^m, \hspace{0.2cm} a_m\in \C  \\
h_i(x) &=\sum_{\langle w,m\rangle=d_i}b_{m,i}x^m, \hspace{0.2cm} b_{m,i}\in \C,\hspace{0.2cm}  i=1,\dots,m
\end{align*}
where each $d_i \in \mathbb{Q}_+$ is the quasihomogeneous degree of $h_i$.
\end{defn}

%The following definition is equivalent to the previous one, but stated in slightly more invariant terms:

As it is obvious from the definition, if $(f,X)$ is a relatively quasihomogeneous pair then there exists an Euler vector field $\chi_w=\sum_{i=1}^nw_ix_i\partial_{x_i}$ with positive rational eigenvalues, $w=(w_1,\dots,w_n)\in \Q^n_+$, such that 
\[\chi_w(f)=f, \quad \chi_w\in \Theta_X.\]
In particular, if $\mu_X(f)<\infty$, then equality $\mu_X(f)=\tau_X(f)$ trivially holds. So it is natural to ask about the validity of the converse implication, i.e. 

\medskip

\noindent ``{\it Does $\mu_X(f)=\tau_X(f)$ imply the relative quasihomogeneity of the pair $(f,X)$?"} 

\medskip

The answer to this relative analog of Saito's theorem is in general negative without further assumptions on the singularity type of the pair $(f,X)$, as one may construct several counter-examples of the following type.

\begin{ex}\label{counter}
Let $f(x,y,z)=x$ and $X=\{xy^3z^3+y^5+z^5=0\}$. Then $\mu_X(f)=\tau_X(f)=1$, and the pair $(f,X)$ is not relatively quasihomogeneous in this coordinate system, and in fact in any coordinate system (since $X$ itself cannot be made quasihomogeneous in any coordinate system). To see this, one may compute a system of generators of the module $\Theta_X\cap \Theta_Y$ and notice that the linear part of any vector field in this module is necessarily a constant multiple of the Euler vector field
\[\chi_{(-1,\smallfrac{1}{5},\smallfrac{1}{5})}=-x\partial_x+\frac{1}{5}y\partial_y+\frac{1}{5}z\partial_z.\]
Since the eigenvalues of any vector field are invariant under changes of coordinates, it follows that there cannot exist another Euler vector field in $\Theta_X\cap \Theta_Y$ with positive rational eigenvalues, which proves the claim (see Corollary \ref{cor-pd} in Section \ref{sec:3}). 
\end{ex}

Despite this fact, there is a wide class of singularities $(f,X)$ where the relative Saito theorem does indeed hold true. Our main results in this direction can be summarised in the following 

\begin{thm}
\label{thm-main}
Let $(f,X)$ be a pair with $\mu_X(f)<\infty$. Suppose also that
\begin{itemize}
\item[(a)] $f\in \m^3$, or
\item[(b)] $X$ is a hypersurface with at most an isolated singularity at the origin.
\end{itemize}
Then, equality $\mu_X(f)=\tau_X(f)$ is equivalent to the relative quasihomogeneity of the pair $(f,X)$.
\end{thm}

The proof of the theorem is a variant of Saito's original proof and will be given in Section \ref{sec:4} with a more precise statement for the range of the weights in each case, related to the multiplicity of $f$. In particular, part (a) of Theorem \ref{thm-main} corresponds to Theorem \ref{thm-main1}, and part (b) to Theorem \ref{thm-main2} respectively. 

The main ingredients of the proof are, apart from Saito's theorem itself, several decomposition formulas for the difference 
\[\mu_X(f)-\tau_X(f)=\dim_{\C}\frac{df(\Theta_X)+\langle f\rangle}{df(\Theta_X)}\] 
presented in Section \ref{sec:2}, as well as an invariant characterization of relative quasihomogeneity in terms of logarithmic vector fields presented in Section \ref{sec:3}. The latter allows us to detect the relative quasihomogeneity of pairs $(f,X)$ by looking solely at the possible eigenvalues of logarithmic vector fields in the intersection $\Theta_X\cap \Theta_Y$.  

We remark finally that one may produce a whole class of counter-examples to the implication ``$\mu_X(f)=\tau_X(f)\Longrightarrow (f,X)$ {\it is relatively quasihomogeneous}", generalising Counter-Example \ref{counter} given above; it consists of pairs $(f,X)$ where $X$ is a non-isolated hypersurface singularity defining an equisingular deformation of a quasihomogeneous isolated singularity $X_0\subset (\C^{n-1},0)$, and $f$ is a linear form corresponding to the deformation parameter. We were not able though to find any other counter-examples away from this class, and in this sense, the validity of the relative Saito theorem in the remaining cases where the multiplicity of $f$ is $\leq 2$ and $X$ is not an isolated hypersurface singularity, is still open.  
 
%Before we close this section it is important to notice that Theorem \ref{thm-main} is not exclusive, in the sense that there are several other cases of pairs $(f,X)$ which verify the relative Saito theorem
%``$\mu_X(f)=\tau_X(f)\Longrightarrow (f,X)$ {\it is relatively quasihomogeneous}", (e.g. when $X=\{x_1\dots x_r=0\}$ is a normal crossing divisor), which we decided not to include them in the present paper.  On the other hand we don't know if Theorem \ref{thm-main} remains true even in the case where $X$ is an {\textsr{ICIS}} of codimension $\geq 2$. Recently though, a formula for

\section{Bruce-Roberts Milnor and Tjurina Numbers}
\label{sec:2}

Let $X \subset (\mathbb{C}^n,0)$ be an analytic variety and $f\in \O_n$ a function germ such that $\mu_X(f)<\infty$.  We denote by $Y=f^{-1}(0)$ the hypersurface defined by $f$. Since $\mu_X(f)$ is finite, then $\mu(f)$ is also finite and thus $Y$ is either a smooth divisor ($\mu(f)=0$) or an isolated hypersurface singularity ($\mu(f)>0$). We also denote by
\[\Theta_Y=\{\delta \in \Theta : \delta(f)\in \langle f\rangle\}\] 
the module of logarithmic vector fields of $Y$, and by
\[H_Y=\ker df(\cdot)=\{\eta \in \Theta : df(\eta)=0\}\]
the submodule of Hamiltonian (or else, Killing) vector fields of $f$, where we denote by $df(\cdot):\Theta \rightarrow \O_n$ the corresponding evaluation map. 
It is easy to verify that in local coordinates $x=(x_1,\dots,x_n)$, the latter module is generated by the derivations
\[\eta_{ij}=\partial_{x_i}f\partial_{x_j}-\partial_{x_j}f\partial_{x_i}, \quad 1\leq i<j\leq n.\]

Finally we denote by
\[\overline{\mu}_X(f)=\dim_\C\frac{\Theta}{\Theta_X+H_Y}\]
and
\[\overline{\tau}_X(f)=\dim_\C\frac{\Theta}{\Theta_X+\Theta_Y}\]
in case where these numbers are finite. Notice that $\overline{\mu}_X(f)\geq \overline{\tau}_X(f)$, and
\begin{equation}
\label{mubar-taubar}
\overline{\mu}_X(f)-\overline{\tau}_X(f)=\dim_{\C}\frac{\Theta_{Y}}{H_Y+(\Theta_{X}\cap \Theta_{Y})}
\end{equation}
where we have used the isomorphism
\[ \frac{\Theta_{X}+\Theta_{Y}}{\Theta_{X}+H_Y}\cong \frac{\Theta_{Y}}{(\Theta_{X}+H_Y)\cap \Theta_{Y}}=
\frac{\Theta_{Y}}{H_Y+(\Theta_{X}\cap \Theta_{Y})}.\]
\begin{prop}\label{prop-1}
Let $(f,X)$ be a pair in $(\C^n,0)$ with $\mu_X(f)<\infty$.
Then the following decomposition formulas hold for the Bruce-Roberts Milnor and Tjurina numbers
\begin{equation}
\label{mubarra}
\mu_{X}(f)=\mu(f)+\overline{\mu}_X(f)
\end{equation}
\begin{equation}
\label{taubarra}
\tau_X(f)=\tau(f)+\overline{\tau}_X(f)
\end{equation}
In particular
\begin{equation}
\label{numinustau2}
\mu_X(f)-\tau_X(f)=\mu(f)-\tau(f)+\overline{\mu}_X(f)-\overline{\tau}_X(f).
\end{equation}
\end{prop}

\begin{proof}
The evaluation map $df:\Theta\to \O_n$ and the inclusion $df(\Theta_X)\subseteq df(\Theta)$ induce the following exact sequences of $\O_n$-modules:
\begin{align}
&\xymatrix@C=0.2cm@R=1.5ex{
0 \ar[rr] &&  \displaystyle \frac{\Theta}{\Theta_{X}+H_Y} \ar[rr]^-{df} &&
\displaystyle\frac{\O_n}{df(\Theta_X)}  \ar[rr]  && \displaystyle\frac{\O_n}{df(\Theta)}  \ar[rr] && 0
}\label{ses-1}\\
&\xymatrix@C=0.2cm@R=1.5ex{
0 \ar[rr] &&  \displaystyle \frac{\Theta}{\Theta_{X}+\Theta_{Y}} \ar[rr]^-{df} &&
\displaystyle\frac{\O_n}{df(\Theta_X)+\langle f\rangle}  \ar[rr]   && \displaystyle\frac{\O_n}{df(\Theta)+\langle f\rangle}     \ar[rr] && 0
}\label{ses-2}
\end{align}
where the respective third morphisms of (\ref{ses-1}) and (\ref{ses-2}) are the natural projections.
The exactness of the above sequences lead to relations (\ref{mubarra}) and (\ref{taubarra}).

Using the $9$-lemma (or by a direct argument) we obtain another short exact sequence defined by the kernels of the natural projections (\ref{ses-1})$\rightarrow$(\ref{ses-2}):
\begin{equation*}
\xymatrix@C=0.2cm@R=1.5ex{
0 \ar[rr] &&  \displaystyle \frac{\Theta_{X}+\Theta_{Y}}{\Theta_{X}+H_Y} \ar[rr]^-{df} &&
\displaystyle\frac{df(\Theta_X)+\langle f\rangle}{df(\Theta_X)} \ar[rr]   &&
\displaystyle\frac{df(\Theta)+\langle f\rangle}{df(\Theta)} \ar[rr] && 0.
}\label{ses-3}
\end{equation*}
which, by (\ref{mubar-taubar}), leads to relation (\ref{numinustau2}). 
\end{proof}

As an immediate corollary of the above we obtain the following algebraic characterization of the equality of the Bruce-Roberts Milnor and Tjurina numbers. 

\begin{cor}\label{caract1mutau}
Let $(f,X)$ be a pair with $\mu_X(f)<\infty$.
Then the following conditions are equivalent
\begin{enumerate}
\item\label{mutau1} $\mu_X(f)=\tau_X(f)$
\item\label{mutau2} $\mu(f)=\tau(f)$ and $\overline{\mu}_X(f)=\overline{\tau}_X(f)$, the latter being equivalent to
\begin{equation}
\label{thetay}
\Theta_Y=\Theta_X\cap \Theta_Y+H_Y.
\end{equation}
\end{enumerate}
\end{cor}

%\begin{rem}
%\label{rem-barra}
%\Theta_Y=H_Y+(\Theta_X\cap \Theta_Y)$
%\end{rem}

In the case where the variety $X$ is also an isolated hypersurface singularity, the formulas obtained above can be substantially improved due to the following

\begin{thm}[\cite{Kour}, see also \cite{NOPT2020}]
\label{Kour} Let $h\in\O_n$ with an isolated singularity at the origin and let $X=h^{-1}(0)$. Let $f\in\O_n$
such that $\mu_X(f)<\infty$. Then
\begin{equation}\label{Keq}
\mu_X(f)=\mu(f)+\mu(f,h)+\mu(h)-\tau(h),
\end{equation}
where $\mu(f,h)$ is the Milnor number of the {\textsc{icis}} (f,h).
\end{thm}

\begin{rem}\label{Kourrem}
We remark that in the case where one has equality $\mu(h)=\tau(h)$ (i.e. $h$ is quasihomogeneous in some coordinate system), formula (\ref{Keq}) above implies the following version of L\^e-Greuel's formula \cite{Gr1}, \cite{Le1}
\[\mu_X(f)=\mu(f)+\mu(f,h)\]
and in particular:
\[df(\Theta_X)=\langle f\rangle +J(f,h)\]
where $J(f,h)=df(H_X)=dh(H_Y)$ is the ideal generated by the $2\times 2$ minors of the Jacobian matrix of the {\textsc{icis}} $(f,h)$. This will be useful in the proof of Theorem \ref{thm-main2} in Section \ref{sec:4}. 
\end{rem}
 
Before we proceed, let us notice that if $X$ is an isolated hypersurface singularity and $f\in \O_n$ such that $\mu_X(f)<\infty$, then we can always choose a function $h\in \O_n$ with an isolated singularity at the origin, such that $X=h^{-1}(0)$ and $\mu_Y(h)<\infty$ as well, where $Y=f^{-1}(0)$ ($h$ will be nothing but a $1$-parameter smoothing of the {\textsc{icis}}  $X\cap Y=(f,h)^{-1}(0)$). In this case we will generally have inequalities
\[\mu_X(f)\neq \mu_Y(f), \quad \tau_X(f)\neq \tau_Y(h)\]
but the following equality
\[\overline{\tau}_X(f)=\overline{\tau}_Y(h)=\dim_{\C}\frac{\Theta}{\Theta_X+\Theta_Y}\]
always holds true, merely by definition (by obvious symmetry in interchanging the roles of $f$ and $h$). We denote this common number by
\[\overline{\tau}(f,h)=\overline{\tau}_X(f)=\overline{\tau}_Y(h).\]

\begin{prop}\label{lemaB}
Let $f,h\in\O_n$ such that $\mu_X(f)<\infty$ and $\mu_Y(h)<\infty$ where $X=h^{-1}(0)$ and $Y=f^{-1}(0)$.
Then
\begin{equation}\label{mumenostau}
\mu_X(f)-\tau_X(f)=\mu(f)-\tau(f)+\mu(h)-\tau(h)+\mu(f,h)-\overline{\tau}(f,h).
\end{equation}
In particular
\begin{equation}\label{symBRDNQ}
\mu_X(f)-\tau_X(f)=\mu_Y(h)-\tau_Y(h).
\end{equation}
\end{prop}

\begin{proof}
By Theorem \ref{Kour} and relation (\ref{taubarra}) we have
\begin{align*}
\mu_X(f)-\tau_X(f)&=\mu(f)+\mu(f,h)+\mu(h)-\tau(h)-\tau_X(f)\\
&=\mu(f)-\tau(f)+\mu(h)-\tau(h)+\mu(f,h)-\overline{\tau}(f,h)\\.
&=\mu_Y(h)-\tau_Y(h)
\end{align*}
\end{proof}

%As an immediate application of the previous result and the definition of $\overline \tau_X(f)$, we obtain the following.

%\begin{cor}\label{qXY}
%Under the conditions of Corollary \textnormal{\ref{lemaB}}, it follows that
%\begin{equation}\label{symBRDNQ}
%\mu_X(f)-\tau_X(f)=\mu_Y(h)-\tau_Y(h).
%\end{equation}
%\end{cor}

%\begin{rem}
%It will follow by Theorem \ref{thm-main} that for a pair $(f,h)$ as above, the difference of the Bruce-Roberts Milnor and Tjurina numbers can be interpreted as the relative degree of non-quasihomogeneity. Relation (\ref{symBRDNQ}) says that this number is an invariant of the pair of hypersurface singularities $(Y,X)$, i.e. it does not depend on choices $(f,h)$ of defining equations.
%\end{rem}

Let us now see how the algebraic characterization of the equality between the Bruce-Roberts Milnor and Tjurina numbers given in Corollary \ref{caract1mutau}, reads in the case of a pair of isolated hypersurface singularities

%\begin{lem}\label{lemaA}
%Let $f,h\in\O_n$ such that $\mu_X(f)<\infty$ and $\mu_Y(h)<\infty$, where $X=h^{-1}(0)$ and $Y=f^{-1}(0)$.
%Then
%\begin{equation}\label{ovtau}
%\mu(f,h)+\mu(f)-\tau(f)\geq \overline{\tau}(f,h).
%\end{equation}

%In particular, if $\mu(f)=\tau(f)$ or $\mu(h)=\tau(h)$, we obtain that
%\begin{equation}\label{ovtau2}
%\mu(f,h)\geq \overline{\tau}(f,h).
%\end{equation}

%\end{lem}

%\begin{proof}
%The finiteness of $\mu_Y(h)$ implies, by \cite[Proposition 2.8]{BiviaRuas}, that the ideal $\J(f,h)+\langle f\rangle$ has finite colength.
%The inclusion $\J(f,h)\subseteq df(\Theta_X)$ shows that
%$$
%\dim_\C\frac{df(\Theta)+\langle f\rangle}{\J(f,h)+\langle f\rangle}\geq  \dim_\C\frac{df(\Theta)+\langle f\rangle}{df(\Theta_X)+\langle f\rangle}=\overline{\tau}(f,h),
%$$
%where the second equality is a direct consequence of (\ref{taubarra}).
%Moreover, let us observe that the colength on the left of the above inequality can be expressed as
%\begin{align*}
%\dim_\C\frac{df(\Theta)+\langle f\rangle}{\J(f,h)+\langle f\rangle}&=\dim_\C\frac{\O_n}{\J(f,h)+\langle f\rangle}-\dim_\C\frac{\O_n}{df(\Theta)+\langle f\rangle}\\
%&=\mu(f,h)+\mu(f)-\tau(f),
%\end{align*}
%where we have applied the L\^e-Greuel formula \cite{Gr1}, \cite{Le1} in the second equality. Hence (\ref{ovtau}) follows.

%By interchanging the roles of $f$ and $h$ in (\ref{ovtau}) we deduce (\ref{ovtau2}) provided that $\mu(f)=\tau(f)$ or $\mu(h)=\tau(h)$.
%\end{proof}

\begin{cor}\label{char2}
Let $f,h\in\O_n$ such that $\mu_X(f)<\infty$ and $\mu_Y(h)<\infty$, where $X=h^{-1}(0)$ and $Y=f^{-1}(0)$.
Then, the following conditions are equivalent
\begin{enumerate}
\item $\mu_X(f)=\tau_X(f)$
\item $\mu(f)=\tau(f)$, $\mu(h)=\tau(h)$ and $\mu(f,h)=\overline{\tau}(f,h)$.
\end{enumerate}
\end{cor}

\begin{proof} It suffices to check $(1) \Rightarrow (2)$.
By Corollary \ref{caract1mutau} the condition $\mu_X(f)=\tau_X(f)$ implies the following two conditions:
$\mu(f)=\tau(f)$, and $\overline{\mu}_X(f)=\overline{\tau}(f,h)$. By (\ref{mumenostau}) they read as
\[\mu(f)=\tau(f) \quad \text{and}\]
\[\mu(h)-\tau(h)+\mu(f,h)=\overline{\tau}(f,h).\]
By (\ref{symBRDNQ}), condition $\mu_X(f)=\tau_X(f)$ is equivalent to $\mu_Y(h)=\tau_Y(h)$ which in turn leads, by Corollary \ref{caract1mutau} again, to $\mu(h)=\tau(h)$, and $\overline{\mu}_Y(h)=\overline{\tau}(f,h)$. By (\ref{mumenostau}), these read as :
\[\mu(h)=\tau(h) \quad \text{and}\]
\[\mu(f)-\tau(f)+\mu(f,h)=\overline{\tau}(f,h).\]
By combining these equations we obtain the required equalities: $\mu(f)=\tau(f)$, $\mu(h)=\tau(h)$ and $\mu(f,h)=\overline{\tau}(f,h)$.
\end{proof}

\begin{rem}
\label{conjbarratau}
It is important to notice that the number $\overline{\tau}(f,h)$ defined above is in general not equal to the ordinary Tjurina number $\tau(f,h)$ of the {\textsc{icis}} $(f,h)$ (c.f. \cite{Looijenga} for exact definition of the latter). Despite this fact, if $\mu_X(f)=\tau_X(f)$ then, as it will follow from the proof of Theorem \ref{thm-main}, the pair $(f,X)$ will be relatively quasihomogeneous, which immediately implies that the {\textsc{icis}} $X\cap Y=(f,h)^{-1}(0)$ will also be quasihomogeneous, and thus:
\begin{equation}
\label{relq-qicis}
\mu_X(f)=\tau_X(f)\Longrightarrow \mu(f,h)=\tau(f,h)=\overline{\tau}(f,h).
\end{equation}
Having this, we conjecture that in general (i.e. without the assumption of relative quasihomogeneity $\mu_X(f)=\tau_X(f)$), the following chain of inequalities will also hold true:
\[\mu(f,h)\geq \tau(f,h)\geq \overline{\tau}(f,h),\] 
where the first inequality is already well known, due to \cite{LS}. 
\end{rem}

Below we give an example which shows that the inverse implication of (\ref{relq-qicis}) does not hold in general, i.e. quasihomogeneity of an {\textsc{icis}} $(f,h)$ does not imply relative quasihomogeneity of the pair $(f,X=h^{-1}(0))$ (nor of $(h,Y=f^{-1}(0))$).

\begin{ex} Let us consider the functions $f,h\in \O_3$ with isolated singularity at the origin given by
$f(x,y,z)=x^2+y^4+z^5$ and $h(x,y,z)=yz+x^3$.
%\begin{align*}
%f(x,y,z)&=x^2+y^4+z^5\\
%h(x,y,z)&=yz+x^3.
%\end{align*}
 Let $X=h^{-1}(0)$ and $Y=f^{-1}(0)$. We have that $f$ is quasihomogeneous with respect to $(1,1/2,2/5)$ and
$h$ is weighted homogeneous with respect to any vector $(w_1,w_2,w_3)\in\Q^3_+$ such that $w_2+w_3=3w_1$.
We have checked that $\mu(f)=12$, $\mu(h)=2$. Moreover
\begin{align*}
\mu(f,h)=10         &&   \mu_X(f)=22  &&\mu_Y(h)=12\\
\tau(f,h)=10          &&  \tau_X(f)=21  &&\tau_Y(h)=11.\\
\overline{\tau}(f,h)=9.
\end{align*}
By the equality $\mu(f,h)=\tau(f,h)=10$ and Vosegaard's theorem \cite{Vo}, it follows that the {\textsc{icis}} $X\cap Y=(f,h)^{-1}(0)$ is quasihomogeneous in an appropriate coordinate system, but the pair $(f,X)$ (and $(h,Y)$ respectively) is not relatively quasihomogeneous in any coordinate system, since $\mu_X(f)-\tau_X(f)=\mu_Y(h)-\tau_Y(h)=1$. Moreover, $\tau(f,h)-\overline{\tau}(f,h)=1$ is positive, as conjecture in  Remark \ref{conjbarratau} above.
\end{ex}

\section{Relative Quasihomogeneity and the Logarithmic Poincar\'e-Dulac Theorem}
\label{sec:3}

In this section we give an invariant characterization of relative quasihomogeneity of pairs $(f,X)$ in terms of logarithmic vector fields. For this we need to pass to the formal category first, where we denote by $\widehat{\O}_n$ the formal completion of $\O_n$. The key lemma is the following logarithmic version of the Poincar\'e-Dulac normal form theorem.

\begin{thm} 
\label{log-pd}
Let $\delta \in \Theta$ be a germ of an analytic vector field at the origin of $\mathbb{C}^n$, such that $\delta(0)=0$. 
Let $\delta_S=\sum_{i=1}^nw_ix_i\partial_{x_i}$ be the semi-simple part of $\delta$,
where $\text{sp}(\delta)=(w_1,\dots,w_n)\in\C^n$ are the eigenvalues of $\delta$ (i.e. of its linear part $j^1\delta$). Then
\begin{itemize}
\item[(a)] there exists a formal change of
coordinates such that
\[\delta=\delta_S+\delta_N\]
where $\delta_N$ is a nilpotent vector field which commutes with $\delta_S$ (and thus with $\delta$)
\[[\delta_S,\delta_N]=0.\] 
\item[(b)] Any eigenfunction $f\in \widehat{\O}_n$ of $\delta$ is also an eigenfunction of its semi-simple part $\delta_S$
\[\delta(f)=f\Longleftrightarrow \delta_S(f)=f, \quad \delta_N(f)=0.\] 
\item[(c)] Any $\delta$-invariant ideal $I\subset \widehat{\O}_n$ is also $\delta_S$ and $\delta_N$-invariant 
\[\delta(I)\subseteq I\Longleftrightarrow \delta_S(I)\subseteq I, \quad \delta_N(I)\subseteq I.\]
\item[(d)] Any $\delta_S$-invariant ideal $I\subset \widehat{\O}_n$ always admits a system of weighted homogeneous generators, i.e. there exist formal series $h_i\in \widehat{\O}_n$, $i=1,\dots,m$, such that
\[I=\langle h_1,\dots,h_m\rangle\]
\[\delta_S(h_i)=d_ih_i, \quad d_i\in \mathbb{C}, \quad i=1,\dots,m.\]  
\item[(e)] Any weighted homogeneous ideal $I=\langle h_1,\dots,h_m\rangle \subset \widehat{\O}_n$ with complex weights $w=(w_1,\dots,w_n)\in \mathbb{C}^n$ and complex weighted degree $d=(d_1,\dots,d_m)\in \mathbb{C}^m$, is also weighted homogeneous with respect to a system of rational weights $w'=(w_1',\dots,w_n')\in \mathbb{Q}^n$, and with rational weighted degree $d'=(d'_1,\dots,d'_{m})\in \mathbb{Q}^m$.  
\end{itemize}
\end{thm}
\begin{proof}
All the parts of the theorem are well known, but scattered across the literature, so we give below the corresponding references:
\begin{itemize}
\item[(a)] This is given in any course of ordinary differential equations, and is known as the Poincar\'e-Dulac normal form c.f. \cite[Ch. 5, \S 23]{Ar}, see also \cite[Satz 3.1]{Sa1}, in terms of derivations. 
\item[(b)] This can be easily proved by using the decomposition of a function in its weighted homogeneous components and the commutativity of the semi-simple with the nilpotent parts of $\delta$, c.f. \cite[Satz 3.2]{Sa1}.
\item[(c)] This is slightly more complicated than (b), c.f. \cite[Theorem 3.2]{Kr} for a recent analytic proof, and also \cite[(2.1)-(2.4)]{SW1} in terms of derivations.  
\item[(d)] This is an exercise on graded submodules of Noetherian graded modules, c.f. \cite[Lemma 3.2]{DamonTop}, and also \cite[(2.4)]{SW1} in terms of derivations.
\item[(e)] This is a straightforward modification of \cite[Lemma 1.4]{Sa1} proved there for a single weighted homogeneous function.
\end{itemize}
\end{proof}

\begin{rem}
In case where the eigenvalues $\text{sp}(\delta)=(w_1,\dots,w_n)\in \mathbb{C}^n$ satisfy certain arithmetic conditions (e.g. they belong in the Poincar\'e domain, c.f. \cite[Ch. 5, \S 24]{Ar}), then both the change of coordinates, as well as the Poincar\'e-Dulac normal form itself, can be chosen to be analytic. In such a case, if $X=V(I)$ is an analytic variety and $f\in \O_n$ is an analytic function germ, then the theorem above says that for any vector field $\delta \in \Theta$, $\delta(0)=0$, the following equivalences hold in analytic Poincar\'e-Dulac coordinates
\[\delta(f)=f\Longleftrightarrow \delta_S(f)=f, \quad \delta_N(f)=0\]
\[\delta \in \Theta_X\Longleftrightarrow \delta_S\in \Theta_X, \quad \delta_N\in \Theta_X.\]
\end{rem}

As an immediate corollary of the above we obtain the following invariant characterization of relative quasihomogeneity for pairs $(f,X)$:
\begin{cor}
\label{cor-pd}
A pair $(f,X)$ is relatively quasihomogeneous if and only if there exists a logarithmic vector field $\delta \in \Theta_X$, $\delta(0)=0$, with positive rational eigenvalues, 
which admits $f$ as an eigenfunction (we can always choose the eigenvalue equal to $1$):
\[\delta \in \Theta_X, \quad \delta(f)=f\]
\[\text{sp}(\delta)=(w_1,\dots,w_n)\in \mathbb{Q}^n_+.\]
\end{cor}
\begin{proof}
Since the eigenvalues of $\delta$ are positive, all parts (a)-(d) in Theorem \ref{log-pd} hold in the analytic category as well. Thus, passing to Poincar\'e-Dulac coordinates $\delta=\delta_S+\delta_N$,  where $\delta_S=\sum_{i=1}^nw_ix_i\partial_{x_i}$, we obtain
\[\delta_S\in \Theta_X, \quad \delta_S(f)=f\]
which is what we wanted to prove. Indeed, by part (d) of Theorem \ref{log-pd} we can find a system of quasihomogeneous generators $\langle h_1,\dots,h_m\rangle=I_X$ of the ideal of functions vanishing on $X$,  so that
\begin{align*}
\delta_S(f)&=f \quad \text{and}\\
\delta_S(h_i)&=d_ih_i, \hspace{0.2cm} d_i\in \Q_+, \hspace{0.2cm} i=1,\dots,m
\end{align*} 
which is equivalent to the relative quasihomogeneity of the pair $(f,X)$ in terms of Definition \ref{def-rel-quasi-1}.
\end{proof}

\section{Relative Saito Theorems}
\label{sec:4}

In this section we prove Theorem \ref{thm-main}. Throughout the section we denote by $\m\subset \O_n$ the maximal ideal of the ring of holomorphic function germs. We start with Part (a) of Theorem \ref{thm-main} with a more precise statement on the range of the weights.

\begin{thm}
\label{thm-main1}
Let $(f,X)$ be a pair with $\mu_X(f)<\infty$, where $X\subset (\C^n,0)$ is an arbitrary analytic set. Suppose also that $f\in \m^3$. Then, condition $\mu_X(f)=\tau_X(f)$ implies the relative quasihomogeneity of the pair $(f,X)$ with respect to a system of weights $w\in ((0,\frac{1}{2})\cap \Q)^n$.
\end{thm}

\begin{proof}
From condition (\ref{mutau2}) in Corollary \ref{caract1mutau} and Saito's theorem \cite[Satz 4.1]{Sa1} applied to $f$, we know that there exists a coordinate system and an Euler vector field $\chi_w \in \Theta_Y$, $\chi_w(f)=f$, with positive rational eigenvalues $w\in \Q^n_+$. In fact, since $f\in \m^3$ we can choose the coordinates such that $w\in ((0,1/2)\cap \mathbb{Q})^n$ (see \cite[Satz 1.3]{Sa1}). Suppose that $\chi_w$ is not tangent to $X$ in these coordinates (or else there is nothing to prove). It follows then by the relation (\ref{thetay}) in Corollary \ref{caract1mutau} that there exists a Hamiltonian vector field $\eta \in H_Y$ of $f$, such that the new vector field $\delta=\chi_w-\eta$ is tangent to $X$, and remains Euler for $f$:
\[\delta=\chi_w-\eta \in \Theta_X\cap \Theta_Y, \quad \delta(f)=\chi_w(f)=f.\]
Denote now by $\text{sp}(\delta)=\text{sp}(\chi_w-\eta)=w'$ the eigenvalues of $\delta$. Since $f\in \m^3$ we have that $j^1\eta=0$ for any $\eta \in H_Y$ (i.e. all Hamiltonian vector fields are non-linear), and thus we obtain the following equality of the corresponding eigenvalues:
\[w'=w\in ((0,\frac{1}{2})\cap \mathbb{Q})^n.\]
The proof is concluded now by Corollary \ref{cor-pd}.
\end{proof}

\begin{rem}
\label{AlekKer}
An alternative proof of Theorem \ref{thm-main1} can be given without using the decomposition $\Theta_Y=\Theta_X\cap \Theta_Y+H_Y$ obtained by the algebraic formula $\overline{\mu}_X(f)=\overline{\tau}_X(f)$, but using instead Alexandrov-Kersken's theorem \cite{Alek}, \cite{Ker}, according to which (c.f. \cite[Proposition 1.2]{Wahl1})
\[\Theta_Y=\langle \chi_w\rangle+H_Y,\]
in appropriate quasihomogeneous coordinates of $f$, where $\chi_w$ is an Euler vector field for $f$, with positive rational eigenvalues $w=(w_1,\dots,w_n)\in ((0,1/2)\cap \mathbb{Q})^n$. Then, since $\mu_X(f)=\tau_X(f)$ implies by definition that there exists $\delta \in \Theta_X\cap \Theta_Y\subset \Theta_Y$, with $\delta(f)=f$, it follows that
\[\delta=\chi_w+\eta\]
for some $\eta \in H_Y$. The rest of the proof follows exactly the same lines.
\end{rem}

Let us prove now Part (b) of Theorem \ref{thm-main} under a more precise statement for the weights.

\begin{thm}
\label{thm-main2}
Let $(f,X)$ be a pair with $\mu_X(f)<\infty$, where $X=h^{-1}(0)\subset (\C^n,0)$ is a hypersurface having at most an isolated singularity at the origin. Then, condition $\mu_X(f)=\tau_X(f)$ implies the relative quasihomogeneity of the pair $(f,X)$ with respect to a system of weights $w\in ((0,1]\cap \Q)^n$.
\end{thm}

\begin{proof}

We split the proof in the following cases, with respect to the multiplicity of $h$.

\medskip

\noindent (1) Case $h\in \m \setminus \m^2$.

\medskip 

\noindent In this case $X=h^{-1}(0)$ is a smooth hypersurface so we can always choose coordinates such that it is weighted homogeneous for any weight system; in particular:
\[h(x,y)=x,\]
so that any vector field $\delta=w_1x\partial_x+\eta$, $\eta \in H_X=\langle \partial_{y_1},\dots,\partial_{y_n}\rangle$ is an Euler vector field for $h$:
\[\delta(h)=w_1h, \quad w_1\in \Q.\]
In these coordinates the function $f$ has the form:
\[f(x,y)=xf_1(x,y)+f_0(y),\]
where either both $f$ and $f_0=f|_{X}$ have an isolated singularity at the origin, or $f$ is regular ($\partial_xf(0)=f_1(0)\ne 0$) and $f_0=f|_{X}$ has an isolated singularity at the origin. It is easy to see now that condition $\mu_X(f)=\tau_X(f)$ implies $\mu(f)=\tau(f)$ and $\mu(f_0)=\tau(f_0)$. Thus, by Saito's theorem applied to $f_0\in \m^2$, we can suppose that after a change of $y$-variables, the function $f_0$ is already a quasihomogeneous polynomial with respect to the weights $w'=(w_2,\dots,w_n)\in ((0,1/2]\cap \mathbb{Q})^{n-1}$, and of quasihomogeneous degree $1$ (see again \cite[Satz 1.3]{Sa1} for the range of the weights). To finish the proof it suffices to show that the weight $w_1 \in \mathbb{Q}$ can be also chosen positive. Indeed, since $f\in \m$ has at most an isolated singularity at the origin, three cases may appear:
\begin{itemize}
\item[(1.a)] the monomial $x$ appears also in the expansion of $f$ and then $w_1=1$, or
\item[(1.b)] the monomial $x^k$, $k\geq 2$, appears in the expansion of $f$, and then $w_1=1/k\in (0,1/2]\cap \mathbb{Q}$, or
\item[(1.c)] a monomial of the form $x^ky_j$ appears in the expansion of $f$ for some $j$ and $k\geq 1$, and then $w_1=\frac{1-w_j}{k} \in (0,1)\cap \mathbb{Q}$.
\end{itemize}

\medskip

\noindent (2) Case $h\in \m^2$.

\medskip

In this case $X=h^{-1}(0)$ also has an isolated singularity at the origin.  The proof splits into the following sub-cases with respect to the multiplicity of $f$.

\medskip

 (2.1) Sub-case $f\in \m^3$. This reduces to Theorem \ref{thm-main1} and there is nothing to prove.

\medskip

(2.2) Sub-case $f\in \m^2 \setminus \m^3$. Working as in the proof of Theorem \ref{thm-main1} we have by Saito's theorem applied to $f$ that there exist coordinates and an Euler vector field $\chi_w\in \Theta_Y$, $\chi_w(f)=f$, with positive rational eigenvalues $w\in ((0,1/2]\cap \mathbb{Q})^n$ (see again \cite[Satz 1.3]{Sa1}). Suppose that in these coordinates the vector field  $\chi_w$ is not tangent to $X$ (else there is nothing to prove).  Then, by relation (\ref{thetay}) in Corollary \ref{caract1mutau} again, there exists  a Hamiltonian vector field $\eta \in H_Y$ for $f$ such that the new vector field $\delta=\chi_w-\eta$ is tangent to $X$ and remains Euler for $f$
\[\delta=\chi_w-\eta \in \Theta_X\cap \Theta_Y, \quad \delta(f)=\chi_w(f)=f.\]
 Let $\text{sp}(\delta)=\text{sp}(\chi_w-\eta)=w'$ be the eigenvalues of $\delta$. Since $f\in \m^2\setminus \m^3$, we have in general $j^1\eta \neq 0$ for $\eta \in H_Y$ (i.e. the set of Hamiltonian vector fields with non-zero linear part is non-empty), and so we may suppose that $w'\neq w$ (else there is again nothing to prove). By Corollary \ref{cor-pd}, it suffices to show that the weights $w'$ can be chosen positive, and in particular $w'\in ((0,1)\cap \mathbb{Q})^n$.  We argue by contradiction by proving the following

\begin{clm}
If in the situation above there exist $k\geq 1$ weights which are non-positive, say $w'_i\leq 0$, $i=1,\dots,k$, then $\mu_X(f)=\infty$.
\end{clm}
\begin{proof}
Indeed, since $f$ has isolated singularity, we have by \cite[Korollar 1.9]{Sa1} that there exist $k$ more weights which are greater or equal to $1$, $w'_{k+i}\geq 1$, $i=1,\dots,k$, and the rest $(n-2k)$-weights belong in the open unit interval $w'_{2k+j}\in (0,1)$, $j=1,\dots,n-2k$, and are situated symmetrically with respect to $1/2$. In particular, working as in the proof of \cite[Lemma 1.10]{Sa1}, we can find a system of weighted homogeneous coordinates $(x,y,z):=(x_1,\dots,x_k,y_1,\dots,y_k,z_1,\dots,z_{n-2k})$ of corresponding degrees
\begin{align*}
\deg_{w'}x_i &=w'_i,\quad i=1,\dots, k\\
\deg_{w'}y_i&=1-w'_i,\quad i=1,\dots,k\\
\deg_{w'}z_j&=w'_{2k+j},\quad j=1,\dots,n-2k
\end{align*}
such that:
\[f(x,y,z)=\sum_{i=1}^kx_iy_i+f_0(z),\]
is weighted homogeneous with respect to the weights:
\[w'=(w'_1,\dots,w'_k,1-w'_1,\dots,1-w'_k,w'_{2k+1},\dots,w'_{n}),\]
of weighted degree $1$, where $f_0\in \frak{m}^2$ is a quasihomogeneous polynomial with respect to the weights $(w'_{2k+1},\dots,w'_n)\in ((0,1)\cap \Q)^{n-2k}$. By Theorem \ref{log-pd} (d), we can choose a weighted homogeneous equation $h=h(x,y,z)$ of the hypersurface $X=h^{-1}(0)$, of weighted degree $\deg_{w'}h=d$.
 Since $h$ also has an isolated singularity at the origin, $d>0$, or else $h\in \langle x\rangle$ will have non-isolated singularity, unless $2k=n$. So, let us suppose first that $2k<n$. Then $h\in \langle y,z\rangle$, so we can write
\[h(x,y,z)=\sum_{i=1}^ky_ih_i(x,y,z)+h_0(z)\]
where the functions $h_i$ and $h_0$ are all weighted homogeneous of degrees $\deg_{w'}h_i=d-1+w'_i$, $\deg_{w'}h_0=d$.
By the fact that $h$ has isolated singularity we see that $h_0\in \m^2$ also has an isolated singularity at the origin $z=0$, and also
\[j^1h_i(x,0,0)=\sum_{j=1}^k\lambda_{ij}x_j, \quad \lambda_{ij}\in \mathbb{C}, \quad i,j=1,\dots,k\]
so that
\[j^2h(x,y,0)=\sum_{i,j=1}^k\lambda_{ij}x_iy_j\]
is a non-degenerate quadratic form. Thus $d=1$.
It follows from this that in the presence of negative weights, the pair $(f,X=h^{-1}(0))$ is reduced to a weighted homogeneous pair necessarily of the same degree $d=1$
\[f(x,y,z)=\sum_{i=1}^kx_iy_i+f_0(z)\]
\[h(x,y,z)=\sum_{i=1}^ky_ih_i(x,y,z)+h_0(z)\]
with $j^1h_i(x,0,0)=\sum_{j=1}^k\lambda_{ij}x_j$.
We can now readily verify that $\mu_X(f)=\infty$, which gives a contradiction.
Indeed, this follows by direct computation of the zero locus of the relative Jacobian ideal which, by Remark \ref{Kourrem}, becomes
\[df(\Theta_X)=\langle f\rangle +J(f,h)\]
where $J(f,h)$ is the ideal generated by the $2\times 2$ minors of the Jacobian matrix of the {\textsc{icis}} $(f,h)$. Then
\[V(df(\Theta_X))=V(\langle f\rangle+J(f,h))=\]
\[=\{y=z=x_ih_j(x,0,0)-x_jh_i(x,0,0)=0, \quad 1\leq i<j\leq k\}\]
which is a determinantal variety defining a curve in the $y=z=0$ space. This proves the claim for the case $2k<n$. Now, if we suppose that $2k=n$ (i.e. that the $z$-variables are missing) and $h\in \langle x\rangle$ is weighted homogeneous of degree $d<0$, then by exactly the same arguments as above, we can reduce the pair $(f,X)$ to the preliminary normal form:
\[f(x,y)=\sum_{i=1}^kx_iy_i\]
\[h(x,y)=\sum_{i=1}^kx_ih_i(x,y)\]
with $j^1h_i(0,y)=\sum_{j=1}^k\lambda_{ij}y_j$, so that 
\[j^2h(x,y)=\sum_{i,j=1}^k\lambda_{ij}x_iy_j\]
is a non-degenerate quadratic form involving only cross-product terms. From this it follows immediately that $\mu_X(f)=\infty$, since
\[V(df(\Theta_X))=\{x=y_ih_j(0,y)-y_jh_i(0,y)=0, \quad 1\leq i<j\leq k\}\]
defines again a curve in the $x=0$ plane. This proves the claim for the case $2k=n$ as well.   
\end{proof}

\medskip

(2.3) Sub-case $f\in \m\setminus \m^2$. This is exactly the same with the case $h\in \m \setminus \m^2$ in the beginning of the proof of the theorem, after interchanging the roles of $f$ and $h$. Indeed, we choose coordinates such that:
\[f(x,y)=x,\]
is weighted homogeneous for any weight system. In these coordinates we can write the function $h$ defining $X=h^{-1}(0)$ as:
\[h(x,y)=xh_1(x,y)+h_0(y).\]
Since $\mu_Y(h)<\infty$ (where $Y=f^{-1}(0)$) we obtain that either both $h$ and $h_0=f|_{Y}$ have an isolated singularity at the origin, or $h$ is regular ($\partial_xh(0)=h_1(0)\ne 0$) and $h_0=h|_{Y}$ has an isolated singularity at the origin. By  Proposition \ref{lemaB} (or equivalently by Corollary \ref{char2}) condition $\mu_X(f)=\tau_X(f)$ is equivalent to $\mu_Y(h)=\tau_Y(h)$ and the latter implies $\mu(h)=\tau(h)$ and $\mu(h_0)=\tau(h_0)$. Thus, by Saito's theorem applied to $h_0 \in \m^2$, we can suppose that after a change of $y$-variables, the function $h_0$ is already a quasihomogeneous polynomial with respect to the weights $w'=(w_2,\dots,w_n)\in ((0,1/2]\cap \mathbb{Q})^{n-1}$, of quasihomogeneous degree $1$. Now we show in the same way as before that the weight $w_1 \in \mathbb{Q}$ can also be chosen positive, namely:
\begin{itemize}
\item[(2.3.a)] either the monomial $x$ appears also in the expansion of $h$ and then $w_1=1$, or
\item[(2.3.b)] the monomial $x^k$, $k\geq 2$, appears in the expansion of $h$, and then $w_1=1/k\in (0,1/2]\cap \mathbb{Q}$, or
\item[(2.3.c)] a monomial of the form $x^ky_j$ appears in the expansion of $h$ for some $j$ and $k\geq 1$, and then $w_1=\frac{1-w_j}{k} \in (0,1)\cap \mathbb{Q}$.
\end{itemize}
This finishes the proof of the theorem.
\end{proof}

\end{document}